\documentclass[english,11pt,reqno]{amsart}

\usepackage{color}
\usepackage{amsthm}
\usepackage{amsmath}
\usepackage{amsfonts}
\usepackage{amstext}
\usepackage[latin1]{inputenc}
\usepackage{amscd}
\usepackage{latexsym}
\usepackage{bm}
\usepackage{amssymb}
\usepackage[all]{xy}
\usepackage{euscript}
\usepackage{a4wide}
\usepackage{amsmath,amssymb,graphicx}
\usepackage{amssymb}
\usepackage{amsmath,boites}
\usepackage{caption}
\usepackage{subcaption}

\newtheorem{theorem}{Theorem}
\newtheorem{proposition}[theorem]{Proposition}

\newtheorem{definition}[theorem]{Definition}

\def\di{\displaystyle}

\newcommand{\N}{\mathbb{N}}
\newcommand{\R}{\mathbb{R}}

\begin{document}
\title[Validating Stochastic Models]{Validating Stochastic Models: Invariance Criteria for Systems of Stochastic Differential Equations and the Selection of a Stochastic Hodgkin-Huxley Type Model}
\author{Jacky Cresson$^{1,2}$, B\'en\'edicte Puig$^1$ and Stefanie Sonner$^{3,4\ *}$}
\thanks{* \textit{corresponding author} \\
Address: BCAM Basque Center for Applied Mathematics, Alameda de Mazarredo 14,  48009 Bilbao, Basque Country, Spain,
Tel.: +34 946 567 842,
Fax: +34 946 567 843,
E-mail: ssonner@bcamath.org}
\maketitle

\vskip 5mm
\begin{tiny}
\begin{enumerate}
\item {Laboratoire de Math\'ematiques Appliqu\'ees de Pau, Universit\'e de Pau et des Pays de l'Adour, avenue de l'Universit\'e, BP 1155, 64013 Pau Cedex, France}

\item {SYRTE UMR CNRS 8630, Observatoire de Paris and University Paris VI, France}

\item {BCAM Basque Center for Applied Mathematics, Alameda de Mazarredo 14,  48009 Bilbao, Basque Country, Spain}

\item {Helmholtz Zentrum M\"unchen, Institut f\"ur Biomathematik und Biometrie, Ingolst\"adter Landstrasse 1, 85764 Neuherberg, Germany}
\end{enumerate}
\end{tiny}
\vskip 5mm

\begin{abstract}
In recent years, many difficulties appeared when taking into account the 
inherent stochastic behavior of neurons and voltage-dependent ion channels in Hodgking-Huxley type models.
In particular, an open problem for a stochastic model of cerebellar granule cell excitability was to ensure
that the values of the gating variables remain within the unit interval. 
In this paper, we provide an answer to this modeling issue and obtain a class of viable stochastic models.
We select the stochastic models thanks to a general criterion for the 
flow invariance of rectangular subsets under systems of stochastic differential equations.
We formulate 
explicit necessary and sufficient conditions, 
that  are valid for both, It\^o's and Stratonovich's interpretation of stochastic differential equations,
improving a previous result obtained by A. Milian [A.Milian, Coll. Math. 1995] in the It\^o case. 
These invariance criteria allow to validate stochastic models in many applications.
To illustrate our results we present numerical simulations for a stochastic Hodgkin-Huxley model. 
\end{abstract}

\section{Introduction}

In recent years, a great deal of activities has been devoted to develop models 
of neuronal excitability that take into account the intrinsic stochastic bioelectrical activity 
of neurons (see \cite{saarinen2008}, \cite{fox}). In \cite{saarinen2008}, the authors apply
It\^o's theory of stochastic differential equations and propose a stochastic model which reproduces 
the irregular electrophysiological activity of an in vitro granule cell (see \cite{saarinen2008} and Figure 6 p.7). 
A particular case of this model is a stochastic version of the classical Hodgkin-Huxley model 
(see \cite{hodgkin-huxley-1952},  \cite{fox}). 
However, as already indicated by the authors, the model suffers severe difficulties 
(\cite{saarinen2008} p.4 and p.10): Undesired values were observed for the gating variables 
that are supposed to take values 
within the unit interval. The solution of 
this modeling difficulty is mentioned as a challenge for future work (\cite{saarinen2008} p.10).
Similar problems also occurred  for the stochastic Hodgkin-Huxley model  in \cite{fox} (see p.2071).
In this article, we provide an answer to this problem and obtain a
family of viable stochastic models for cerebellar granule cell excitability. 
The admissible models are derived from 
a general invariance theorem for systems of stochastic differential equations.\\

We formulate invariance results 
 in a general setting that  allow to validate 
stochastic models in many applications.
When
the solutions of a given system of stochastic differential equations  describe 
quantities that necessarily take values within a certain range, the  
problem can be mathematically analyzed by studying the flow invariance of 
rectangular subsets of the euclidean space. 
To be more precise, we consider systems of It\^o differential equations of the form
$$dX(t)=f(t,X(t)) dt+g(t,X(t))dW(t),
$$
where the process $X$ is vector-valued. For the concrete formulation we refer to Section \ref{main-results}.
We characterize the class of functions $f$ and $g$ that 
lead to viable stochastic models and formulate explicit necessary and sufficient conditions that can be directly checked and easily verified in applications.
We further show that the invariance theorems are valid for 
both,  It\^o's and Stratonovich's interpretation of stochastic differential equations. 
The result for the It\^o case was previously obtained by A. Milian in \cite{milian-1995}. 
As discussed in \cite{turelli-1977}, in a concrete application 
it is generally not easy to decide which interpretation should be applied. 
Our results show that the qualitative behavior of solutions regarding non-negativity and boundedness 
is independent  
of It\^o's or Stratonovich's interpretation. 
Other properties of the solutions, however, may strongly depend on the choice of the interpretation (see \cite{oksendal-2003}).  
\\

The outline of our paper is as follows:
In Section \ref{main-results} we introduce the class of stochastic systems we study and formulate general 
invariance criteria for systems of stochastic differential equations. 
We then apply the results  to obtain viable stochastic models for cerebellar granule cell excitability in Section \ref{application}. 
In Section \ref{simulations} we present numerical simulations to 
illustrate the model behaviour.
Finally, in Section \ref{proofs} we recall the  results obtained by A. Milian in \cite{milian-1995} and 
present the proofs of the invariance theorems.
\\

\section{Invariance Criteria for Stochastic Differential Equations}\label{main-results}

Let $(\Omega , \mathcal{F},P)$ be  a probability space with a 
right-continuous increasing family $F=(\mathcal{F}_t )_{t\geq 0}$ of sub-$\sigma$-fields of $\mathcal{F}$ each containing all sets of $P$-measure zero. We consider systems of stochastic It\^o equations of the form
\begin{equation}
\label{stochastic}
X(t) = X_0 +\di\int_{t_0}^t f(s,X(s) ) ds +\int_{t_0}^t g(s,X(s) ) dW(s), \qquad\quad t\in[t_0,\infty[,
\end{equation}
where $f=[f_i]:[0,\infty[\times \R^m\rightarrow \R^m$ is Borel-measurable, and 
$g=[g_{ij}]:[0,\infty[\times \R^m\rightarrow \R^{m\times r}$ is a Borel-measurable mapping into the set of all $\R^{m\times r}$-matrices,
$i=1,\dots,m, j=1,\dots,r$. 
Furthermore, $W : [0,\infty[\times \Omega \rightarrow \R^r$ denotes an $r$-dimensional $F$-Wiener process, 
the initial time $t_0$ is non-negative and $X_0\in\R^m$
is the given initial data. \\

The stochastic  integral equations (\ref{stochastic}) are commonly written as system of stochastic It\^o differential equations, 
\begin{eqnarray*}
dX(t) &=&  f(t,X(t) ) dt + g(t,X(t) ) dW(t), \qquad\quad t\in[t_0,\infty[,\\
X(t_0)&=&X_0,
\end{eqnarray*}
where the function $g$ represents the stochastic perturbation and $f$ the deterministic part. 
Indeed, if $g\equiv0$ we obtain the corresponding unperturbed 
deterministic system of ODEs. 
\\

In the sequel, we denote by $(f,g)$ stochastic initial value problems of the form (\ref{stochastic}). 
We aim at formulating explicit necessary and sufficient conditions on the 
functions $f$ and $g$ for the non-negativity and boundedness of 
solutions. 
The conditions can directly be verified and allow to explicitly characterize 
the class of admissible models in applications. In Section \ref{proofs} we deduce our main theorems from a 
more general result about the stochastic invariance of polyhedral subsets of $\R^m$. 
However, in applications the non-negativity and boundedness of solutions are the most relevant 
modeling issues.\\

Since our aim is not to establish the well-posedness of the stochastic initial value problem 
but to study the qualitative behavior of solutions, we assume that  for every
initial time $t_0\geq 0$ and initial data $X_0\in\R^m$ there exists
a unique solution of the stochastic problem  (\ref{stochastic}).
\\

\begin{definition}
We say that the subset $K\subset \R^m$ is \textbf{ invariant} for the stochastic system $(f,g)$ 
if for every initial data $X_0\in K$ and initial time $t_0\geq 0$ the corresponding solution $X(t)$, $t\geq t_0$, satisfies
$$P\left(\left\{X(t)\in K,\ t\in[t_0,\infty[\right\}\right)=1.
$$
\end{definition}
\vspace*{3ex}

The following theorem characterizes the class of functions $f$ and $g$ such that the stochastic system $(f,g)$
preserves the positivity of solutions. That is, solutions emanating from non-negative initial 
data (almost surely) remain non-negative as long as they exist. \\

\begin{theorem}\label{mainA}
Let $I\subset\{1,\dots,m\}$ be a non-empty subset. Then, the 
set 
$$
K^+:=\{x=(x_1,\dots,x_m)\in\R^m, \ x_i\geq 0,\, i\in I\}
$$ 
is invariant for the
stochastic system $(f,g)$ if and only if
\begin{eqnarray*}
f_i (t,x) &\geq & 0 \qquad \textnormal{\emph{for}}\  x\in K^+\ \textnormal{\emph{such that}}\ x_i=0, \nonumber \\
g_{i,j} (t,x) &= &0 \qquad  \textnormal{\emph{for}}\  x\in K^+\ \textnormal{\emph{such that}}\ x_i=0,\  j=1,\dots,r,\nonumber
\end{eqnarray*} 
for all $t\geq 0$ and $i\in I$. 

This result applies independent of It\^o's or Stratonovich's interpretation of stochastic differential equations.
\end{theorem}
\vspace*{2ex}

The solutions of mathematical models often describe quantities that necessarily take values within a particular range. 
We next formulate a criterion for 
the invariance of rectangular subsets
of the phase space $\R^m$.

\begin{theorem}\label{mainB}
Let $I\subset\{1,\dots,m\}$ be a non-empty subset and $a_i,b_i\in\R$ such that $b_i>a_i$. Then, the 
set 
$$
K:=\{x\in\R^m :\ a_i\leq x_i\leq b_i,\, i\in I\}
$$ 
is invariant for the
stochastic system $(f,g)$ if and only if
\begin{eqnarray}
f_i (t,x) &\geq & 0 \qquad \textnormal{\emph{for}}\  x\in K\ \textnormal{\emph{such that}}\ x_i=a_i, \nonumber \\
f_i (t,x) &\leq & 0  \qquad   \textnormal{\emph{for}}\  x\in K\ \textnormal{\emph{such that}}\ x_i=b_i,  \nonumber \\
g_{i,j} (t,x) &= &0 \qquad  \textnormal{\emph{for}}\  x\in K\ \textnormal{\emph{such that}}\ x_i\in\{a_i,b_i\},\  j=1,\dots,r,\nonumber
\end{eqnarray}
for all $t\geq 0$ and $i\in I$. 

This result is valid independent of It\^o's or Stratonovich's interpretation.
\end{theorem}

If we apply Theorem \ref{mainA} or Theorem \ref{mainB}
to the corresponding unperturbed deterministic system $(f,0)$ we recover the well-known tangential 
condition for systems of ODEs, which is necessary and sufficient for the flow invariance of subsets
of  $\R^m$ (see \cite{Walter} or \cite{Pavel}).\\

Finally, we formulate a criterion for the validity of comparison principles for the solutions of stochastic systems.
A. Milian stated the following theorem for systems of It\^o equations in \cite{milian-1995}. We recall her result and 
show that it remains valid if we apply Stratonovich's interpretation of stochastic differential equations.\\

\begin{theorem}\label{mainC}
Let $I$ be a non-empty subset of $\{1,\dots,m\}$. We assume that $(f,g)$ and $(\tilde f,\tilde g)$ are
stochastic systems of the form $(\ref{stochastic})$ with   given initial data 
 $X_0,Y_0\in\R^m$, and denote by $X$ and $Y$ the corresponding solutions.
Then, the following
statements are equivalent:
\begin{itemize}
\item[(a)] For all $t_0\geq 0$ and $i\in I$, if the initial data satisfy $(X_0)_i\geq (Y_0)_i$, then
$$P\left(\left\{X_i(t)\geq Y_i(t),\ t\in[t_0,\infty[, i\in I\right\}\right)=1.
$$

\item[(b)] The functions $f, \tilde f,g$ and $\tilde g$ satisfy
\begin{eqnarray*}
f_i(t,x)&\geq &\tilde f_i(t,y),\ \quad \qquad t\geq 0,\\
g_{ij}(t,x)&=&\tilde g_{ij}(t,y),\quad \qquad t\geq 0,\ j=1,\dots,r,
\end{eqnarray*}
for all  $i \in I$ and  $x,y\in\R^m$ such that  $x_i=y_i$ and $x_k\geq y_k$ for $k\in I$.\\
\end{itemize}
 \end{theorem}

\section{The Validation of a Stochastic Hodgkin-Huxley Type Model}\label{application}

A stochastic model for  cerebellar granule cell excitability was proposed and numerically studied 
in \cite{saarinen2008}.
As many biophysical models of neurons it is based
on the well-known deterministic Hodgkin-Huxley formalism \cite{hodgkin-huxley-1952}, which qualitatively describes the 
conduction and excitation in nerves.
Such models are commonly formulated as systems of deterministic ODEs.
The behavior of neurons and voltage-dependent ion channels, however, is known to be stochastic in nature,
which motivates the modeling approach in \cite{saarinen2008}.
The mathematical model  is formulated as system of stochastic differential equations for the dependent model variables 
$x_i$, which represent the gating variables for the specific ion channels, 
the transmembrane potential $V$ and the intracellular calcium concentration $C$,
\begin{align}\label{model}
dx_i & =  f_i (V,x_i ) dt +\sigma_i\, dW_{i}(t) , &&i=1,\dots ,8,\\
dx_9& =  f_9 (V,C,x_9 ) dt +\sigma_9\, dW_9(t) ,\nonumber\\
dV & =  F(t,V,x_1,\dots,x_9) dt,\nonumber\\
dC & =  G(V,C,x_7,x_8) dt,\nonumber
\end{align}
where  the reaction functions in the equations for the gating variables are given by
\begin{align*}
&f_i (V,x_i )\ =\alpha_i (V) (1-x_i ) -\beta_i (V) x_i, && i=1,\dots,8,\\
&f_9(V,C,x_9)= \alpha_9 (V,C) (1-x_9) -\beta_9 (V,C) x_9,
\end{align*}
and the rate functions for activation $\alpha_i$ and inactivation $\beta_i$ are continuous and positive.  The stochastic differential equations are 
interpreted in the sense of It\^o,
$W_i(t), t\geq 0,$ denote standard scalar Wiener processes, $dW_i$ the corresponding It\^o differentials,
and the parameters $\sigma_i$ are positive and constant, $i=1,\dots,9$.
For the concrete form of the interaction functions $F$ and $G$ and the complete description of the model we refer to \cite{saarinen2008}. 
\\

This model extends a previous deterministic model for cerebellar granule cell excitability by adding the 
stochastic terms $\sigma_i\,dW_i(t)$ in the governing equations for the gating variables $x_i$, $i=1,\dots,9$. 
Ion channel stochasticity has been detected experimentally and is due to the thermal interaction of molecules constituting an ion channel. 
It can be observed as random opening and closing of an ion channel at an experimentally fixed membrane potential (see \cite{saarinen2008} and also \cite{fox}).\\

The gating variables $x_i$ describe the opening and closing rates of the specific ion channels and necessarily take values 
within the interval $[0,1]$.
While the corresponding unperturbed deterministic model, where $\sigma_i=0, i=1,\dots,9$, certainly ensures this property, 
it cannot be guaranteed by the stochastic model (\ref{model}):

The parameters $\sigma_i$, $i=1,\dots,9$, which take into account the intensity of the stochastic perturbations, 
were taken to be constant in the model and the simulations presented in \cite{saarinen2008}. The necessity 
to carefully choose these parameters was indicated. 
In particular, 
undesired values of the gating variables were observed and discussed, it was stressed that this modeling 
issue needed to be solved and highlighted as a challenge for future work
(see \cite{saarinen2008}, p.4 and p.10).
Similar difficulties also occurred for the stochastic Hodgkin-Huxley model developed in \cite{fox} (see p.2071).
Our results show that independent of the choice of the parameters $\sigma_i$
the invariance of the unit interval 
cannot be guaranteed by the stochastic model (\ref{model}) if we take these parameters to be constant. 
Indeed, the conditions on the stochastic perturbations in Theorem \ref{mainB} applied to the 
model (\ref{model}) and the invariant subset 
$$
\widetilde K=\{y\in\R^{11},\  0\leq y_i\leq 1,\, i=1,\dots,9\}
$$ 
are never satisfied.\\

We obtain viable stochastic models if we
replace the constants $\sigma_i$ by appropriate functions $g_i$, that ensure the desired invariance of the unit interval.
To be more precise, we propose to consider models of the form
\begin{align}\label{model2}
dx_i & =  f_i (V,x_i ) dt +g_i(t,V,C,x)\, dW_{i}(t) , &&i=1,\dots ,8,\\
dx_9& =  f_9 (V,C,x_9 ) dt +g_9(t,V,C,x)\, dW_9(t) ,\nonumber\\
dV & =  F(t,V,x) dt,\nonumber\\
dC & =  G(V,C,x_7,x_8) dt,\nonumber
\end{align}
where $x=(x_1,\dots,x_9)$, and the stochastic perturbations $g_i:[0,\infty[\times \R^{11}\rightarrow\R$ satisfy
$$g_i(t,y)=0\qquad\quad\textnormal{for}\ y\in\widetilde K\ \textnormal{such that}\ y_i\in\{0,1\}, 
$$
for all $t\geq 0$ and $i=1,\dots,9$.
\\

\begin{proposition}
The modified stochastic model (\ref{model2})
ensures that the gating variables $x_i$ take values within the interval $[0,1]$, for all $i=1,\dots,9$. This is valid for 
It\^o's and for Stratonovich's interpretation of the stochastic differential equations.  
\end{proposition}

\begin{proof}
The statement is a direct consequence of Theorem \ref{mainB} since the interaction functions $f_i$ and 
stochastic perturbations $g_i$ in the governing equations for the gating variables $x_i$, $i=1,\dots, 9$, in the modified stochastic
model (\ref{model2}) satisfy the required conditions.
\end{proof}

One possible choice for the stochastic perturbations are functions of the form
$$
g_i(x_i)=\sigma_i x_i (1-x_i),
$$
with constants $\sigma_i\in\R^+$, $ i=1,\dots,9.$
\\

\section{Numerical Simulations}\label{simulations}

To illustrate our results we present numerical simulations for a simplified 
version of the stochastic model discussed in the previous section and 
consider a stochastic version of the classical Hodgkin-Huxley model \cite{hodgkin-huxley-1952}
Despite its simplicity, the deterministic Hodgkin-Huxley model has always been playing a 
very important role in the study of neuron excitability (\cite{Meunier}).
However, stochasticity should be included in the model to 
take into account the stochastic behavior of the ion channel kinetics
(see \cite{Meunier}, p.558 and p.559).

\begin{figure}[h]
        \begin{subfigure}[b]{0.4\textwidth}
                \centering
                \includegraphics[width=7cm]{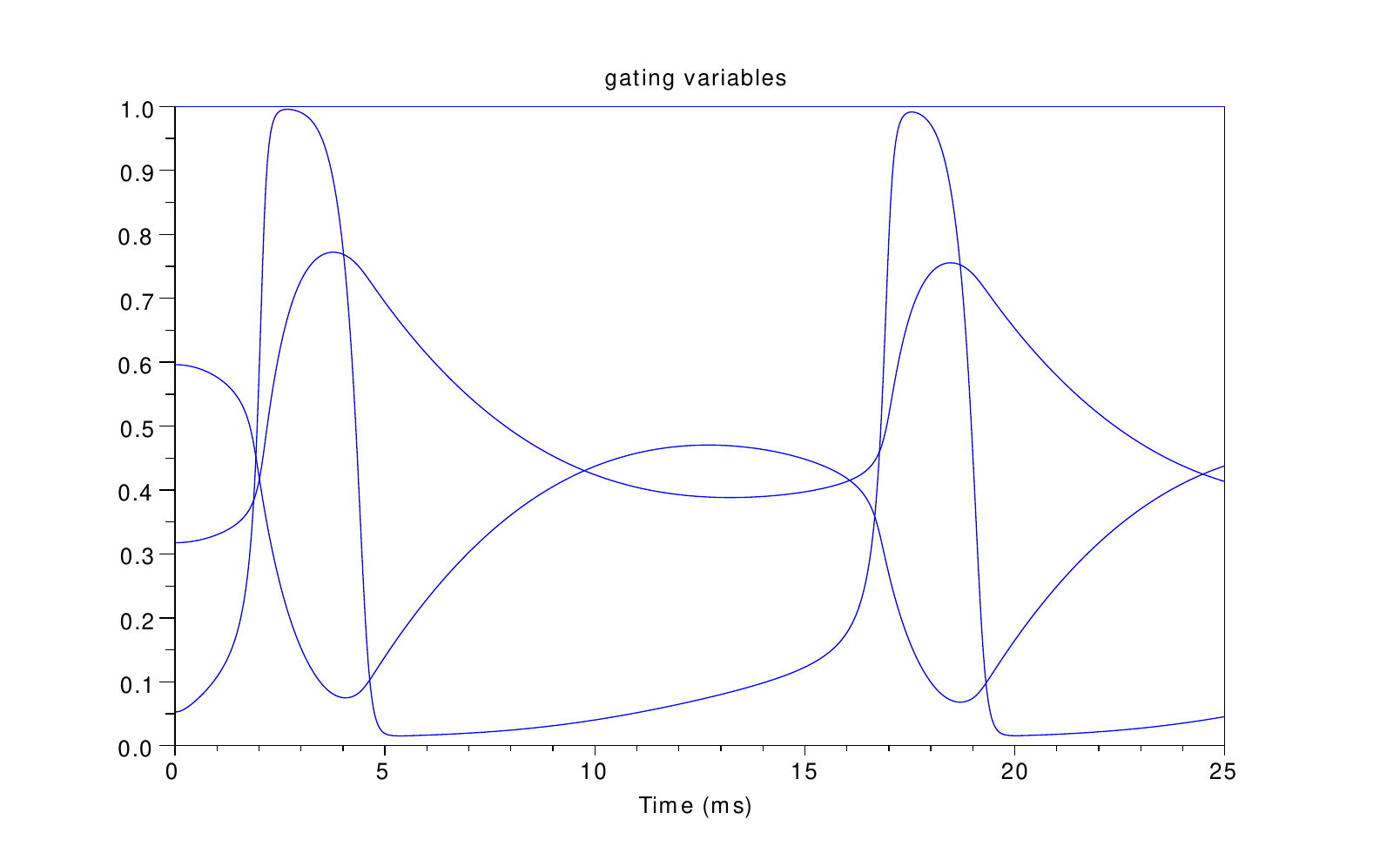}
                \caption{Gating Variables}
        \end{subfigure}%
        \quad
             \begin{subfigure}[b]{0.4\textwidth}
                \centering
                \includegraphics[width=7cm]{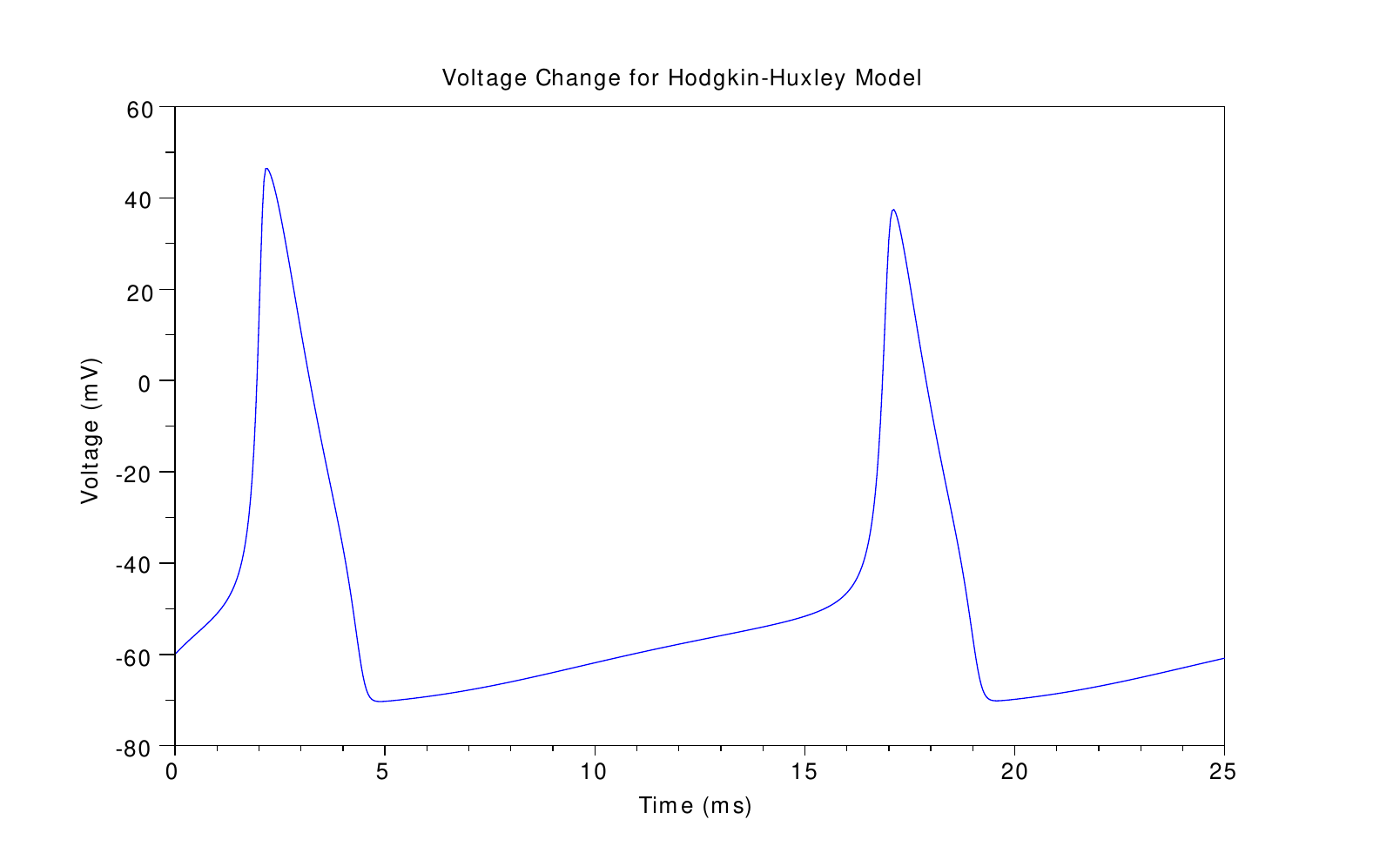}
                \caption{Voltage}
        \end{subfigure}                
        \caption{Deterministic Model}\label{fig1}
\end{figure}

The model we consider is  formulated as system of ordinary differential equations for the dependent model variables 
$x_i$, $i=1,2,3$, that represent the gating variables for the specific ion channels, and
the voltage $V$,
\begin{align}\label{model2}
\frac{dx_i}{dt} & = \alpha_i(V)(1-x_i)-\beta_i(V)x_i,\\
\frac{dV}{dt} & = \frac{1}{C}[I-g_{Na}x_1^3x_3(V-E_{Na})-g_K x_2^4(V-E_K)-g_L (V-E_L)]. \nonumber
\end{align}
The rate functions for activation and inactivation are given by
\begin{align*}
\alpha_1(V)&=\frac{0.1(V+35)}{1-\exp(-\frac{V+35}{10})},&&\beta_1(V)=4.0\exp(-0.0556(V+60)),\\
\alpha_2(V)&=\frac{0.01(V+50)}{1-\exp(-\frac{V+50}{10})},&&\beta_2(V)=0.125\exp(-\frac{V+60}{80}),\\
\alpha_3(V)&=0.07\exp(-0.05(V+60)),&&\beta_3(V)=\frac{1}{1+\exp(-0.1(V+30))},
\end{align*}
 and the parameter values by
 \begin{align*}
C & =0.01\frac{\mu F}{cm^2},      &&  g_{Na}=1.2\frac{mS}{cm^2},&&     g_K=0.36\frac{mS}{cm^2},  &&g_L=0.03\frac{mS}{cm^2}, \\
I&=0.1\ mV,    &&E_{Na}=55.17\ mV,      &&  E_K=-72.14\ mV,   &&    E_L=-49.42\ mV,
\end{align*}
(see \cite{hodgkin-huxley-1952}).

\begin{figure}[h]
           \begin{subfigure}[b]{0.4\textwidth}
                \centering
                \includegraphics[width=7cm]{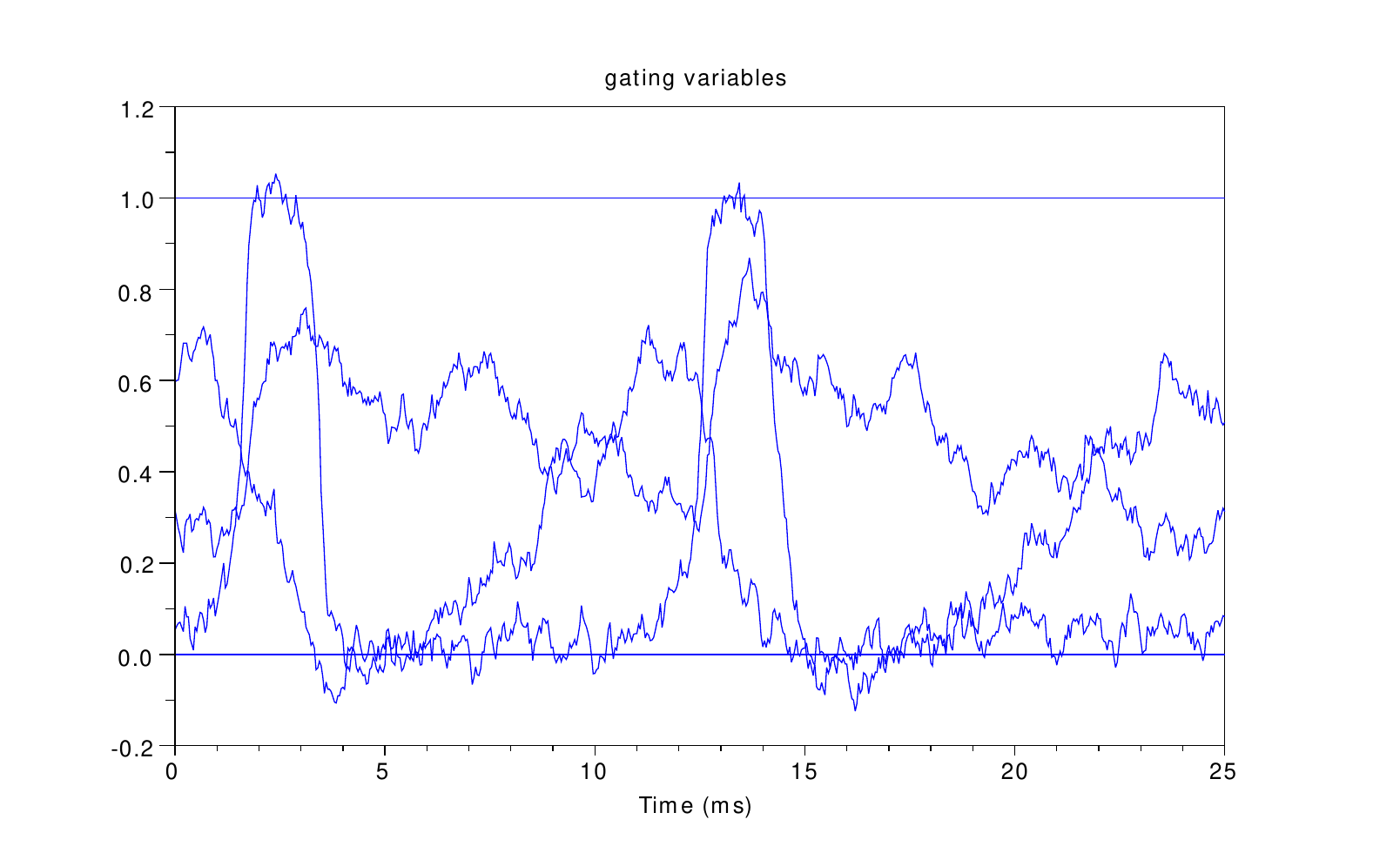}
                \caption{Gating Variables: $\sigma=0.1$}
                \end{subfigure}
                \quad
      \begin{subfigure}[b]{0.4\textwidth}
                \centering
                \includegraphics[width=7cm]{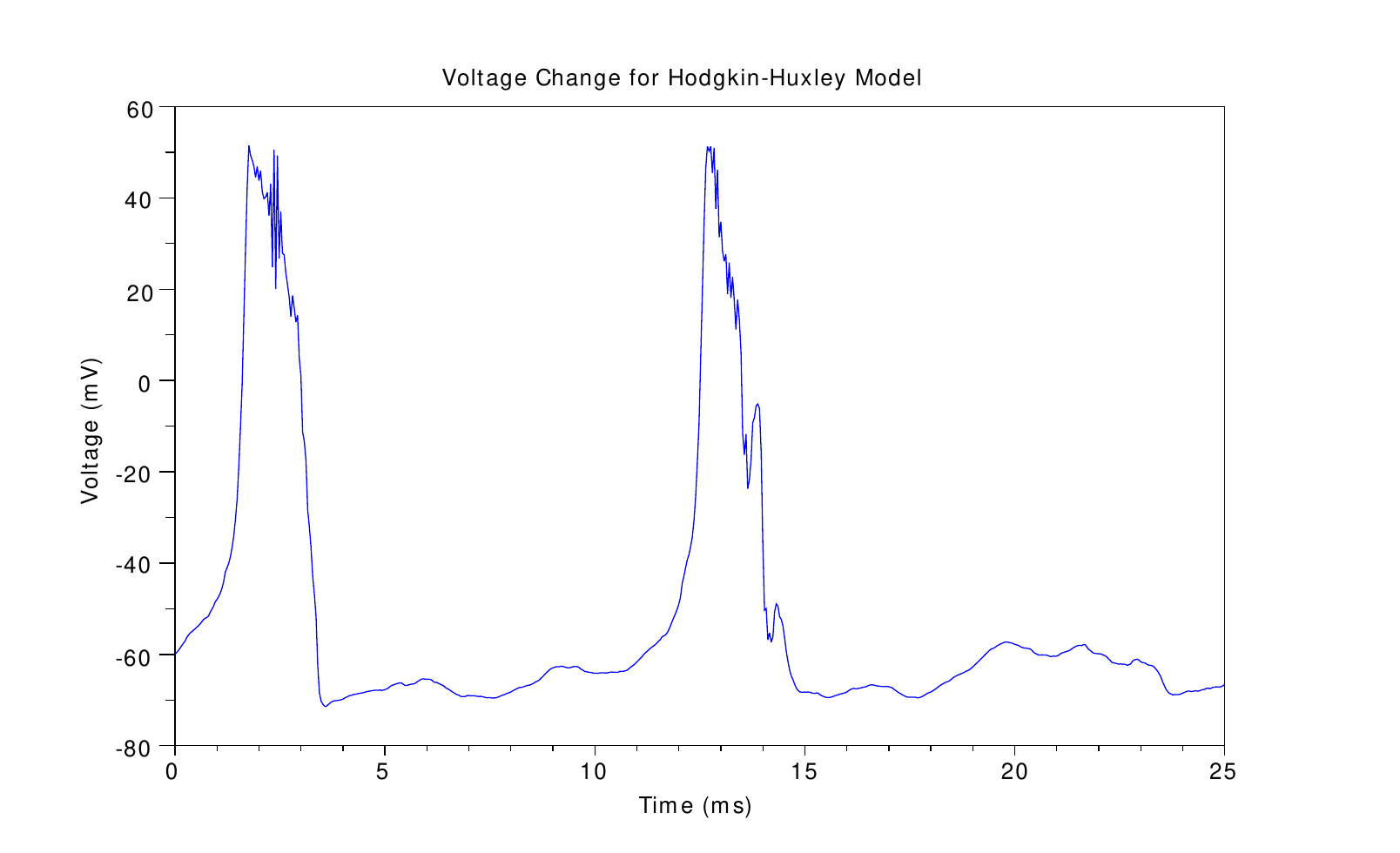}
                \caption{Voltage: $\sigma=0.1$}
        \end{subfigure}\\
                  \begin{subfigure}[b]{0.4\textwidth}
                \centering
                \includegraphics[width=7cm]{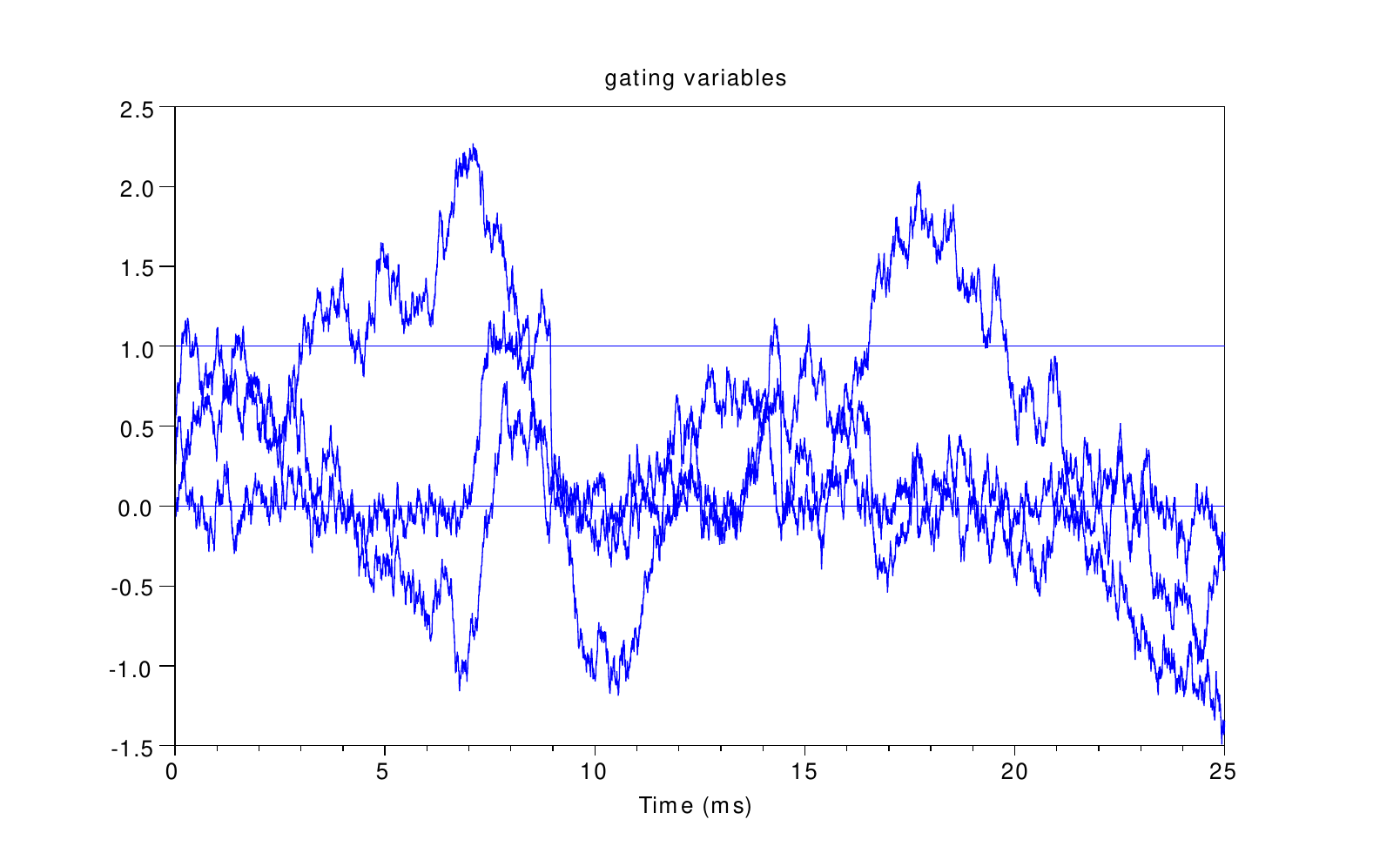}
                \caption{Gating Variables: $\sigma=0.5$}
        \end{subfigure}%
        \quad
            \begin{subfigure}[b]{0.4\textwidth}
                \centering
                \includegraphics[width=7cm]{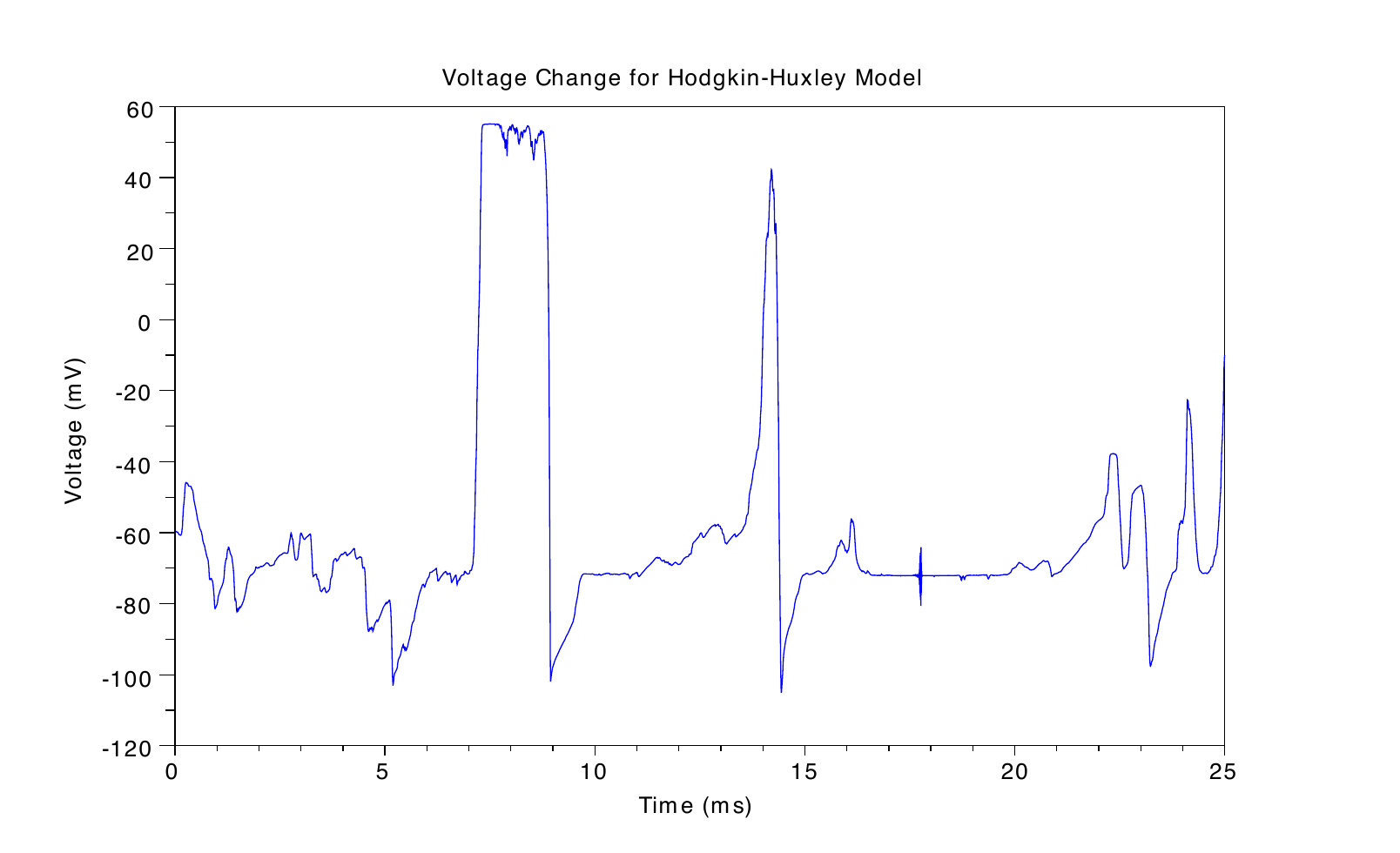}
                \caption{Voltage: $\sigma=0.5$}
                \end{subfigure}\\
                \caption{Stochastic Model: Additive Noise}\label{fig2}
\end{figure}

We illustrate the model behavior in Figure \ref{fig1}. 
The gating variables $x_i$, $i=1,2,3$, describe the opening and closing rates of the specific ion channels and necessarily take values 
within the interval $[0,1]$.
The deterministic model (\ref{model2}) certainly ensures this property.
\\

Following the modeling approach in \cite{saarinen2008}, we may
extend the deterministic model  by adding the 
stochastic terms $\sigma_i\,dW_i(t)$ in the governing equations for the 
gating variables $x_i$ in the model (\ref{model2}),
which leads to the system of stochastic differential equations 
\begin{align}\label{model3}
dx_i(t) & =   (\alpha_i(V(t))(1-x_i(t))-\beta_i(V(t))x_i(t)) dt +\sigma_i\, dW_{i}(t),\qquad i=1,2,3.
\end{align}
We interpret the stochastic differential equations 
in the sense of It\^o,
$W_i(t), t\geq 0,$ denote standard scalar Wiener processes, $dW_i$ the corresponding It\^o differentials,
and the parameters $\sigma_i$ are positive and constant, $i=1,\dots,3$.
\\

\begin{figure}[h]
        \begin{subfigure}[b]{0.4\textwidth}
                \centering
                \includegraphics[width=7cm]{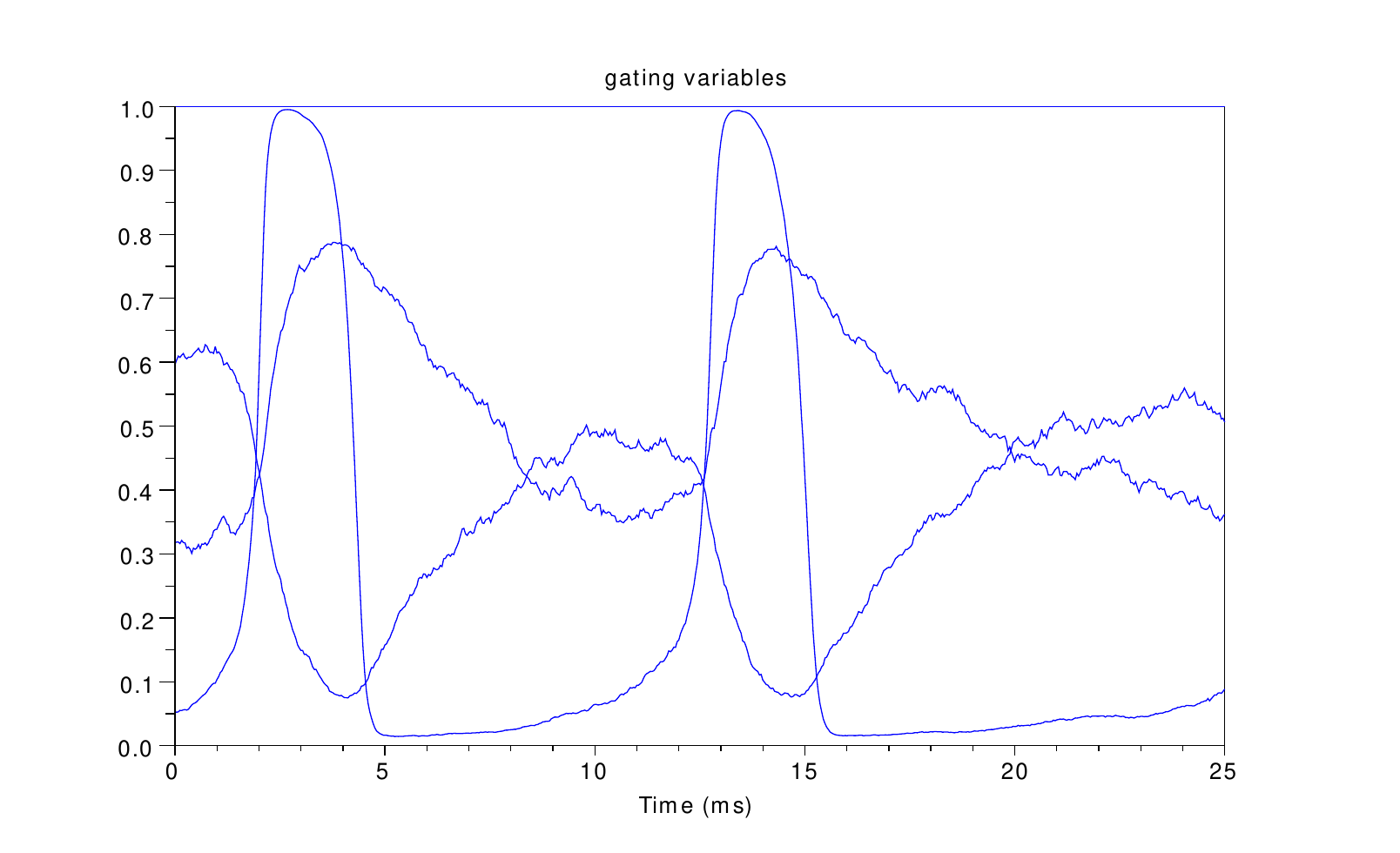}
                \caption{Gating Variables: $\sigma=0.1$}
        \end{subfigure}%
        \quad
        \begin{subfigure}[b]{0.4\textwidth}
                \centering
                \includegraphics[width=7cm]{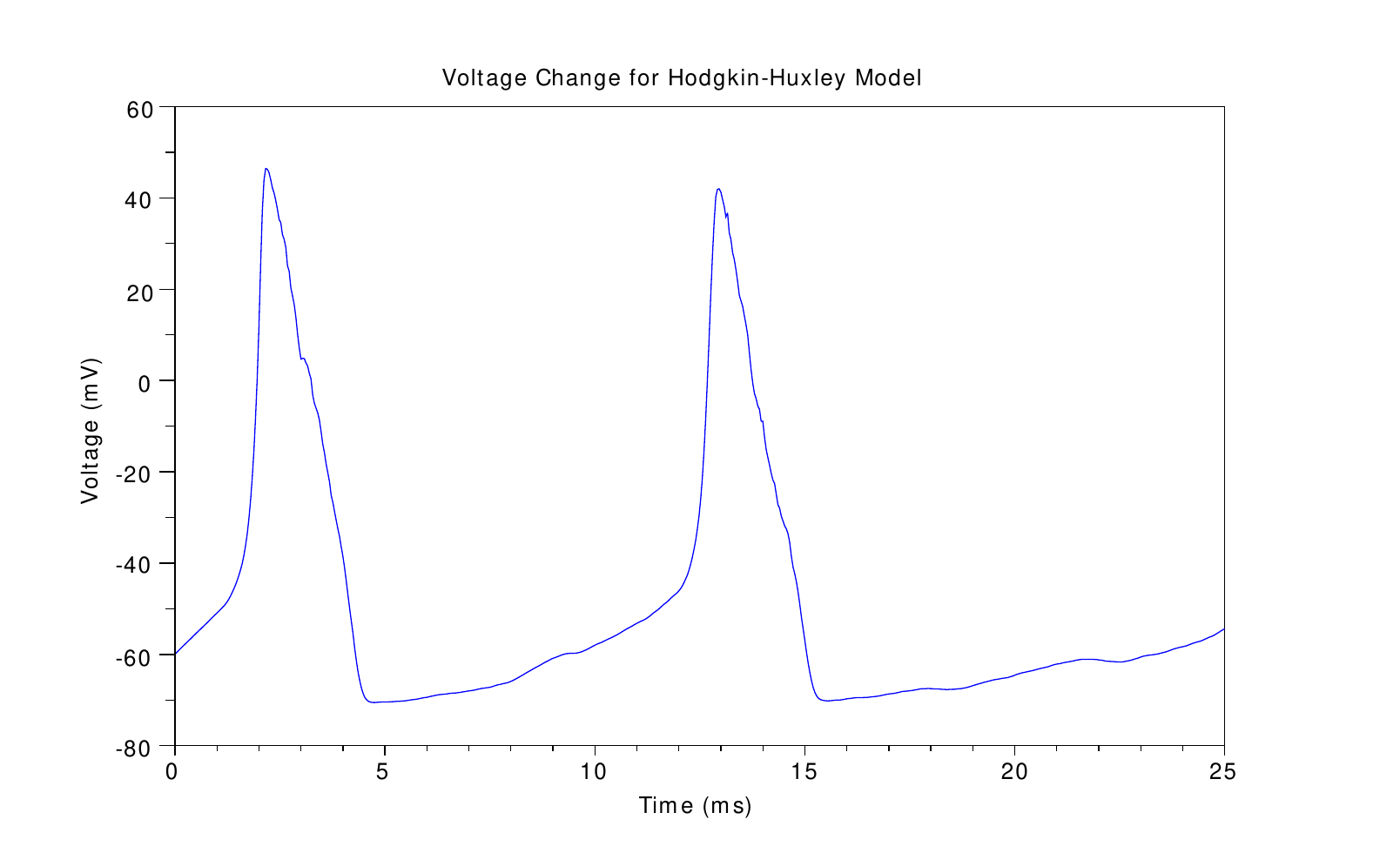}
                \caption{Voltage: $\sigma=0.1$}
        \end{subfigure}\\
     \begin{subfigure}[b]{0.4\textwidth}
                \centering
                \includegraphics[width=7cm]{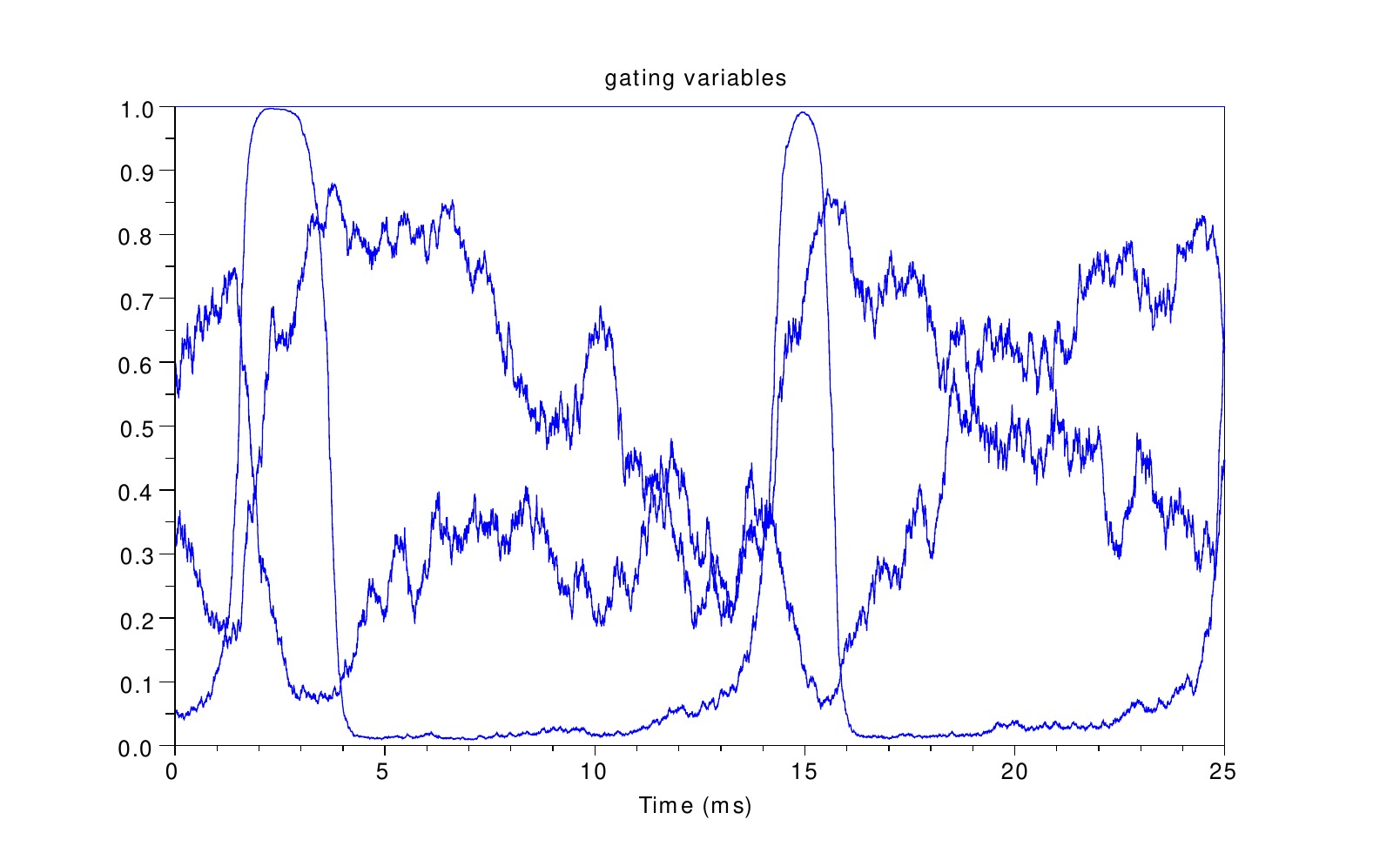}
                \caption{Gating Variables: $\sigma=0.5$}
                \end{subfigure}
    \quad
      \begin{subfigure}[b]{0.4\textwidth}
                \centering
                \includegraphics[width=7cm]{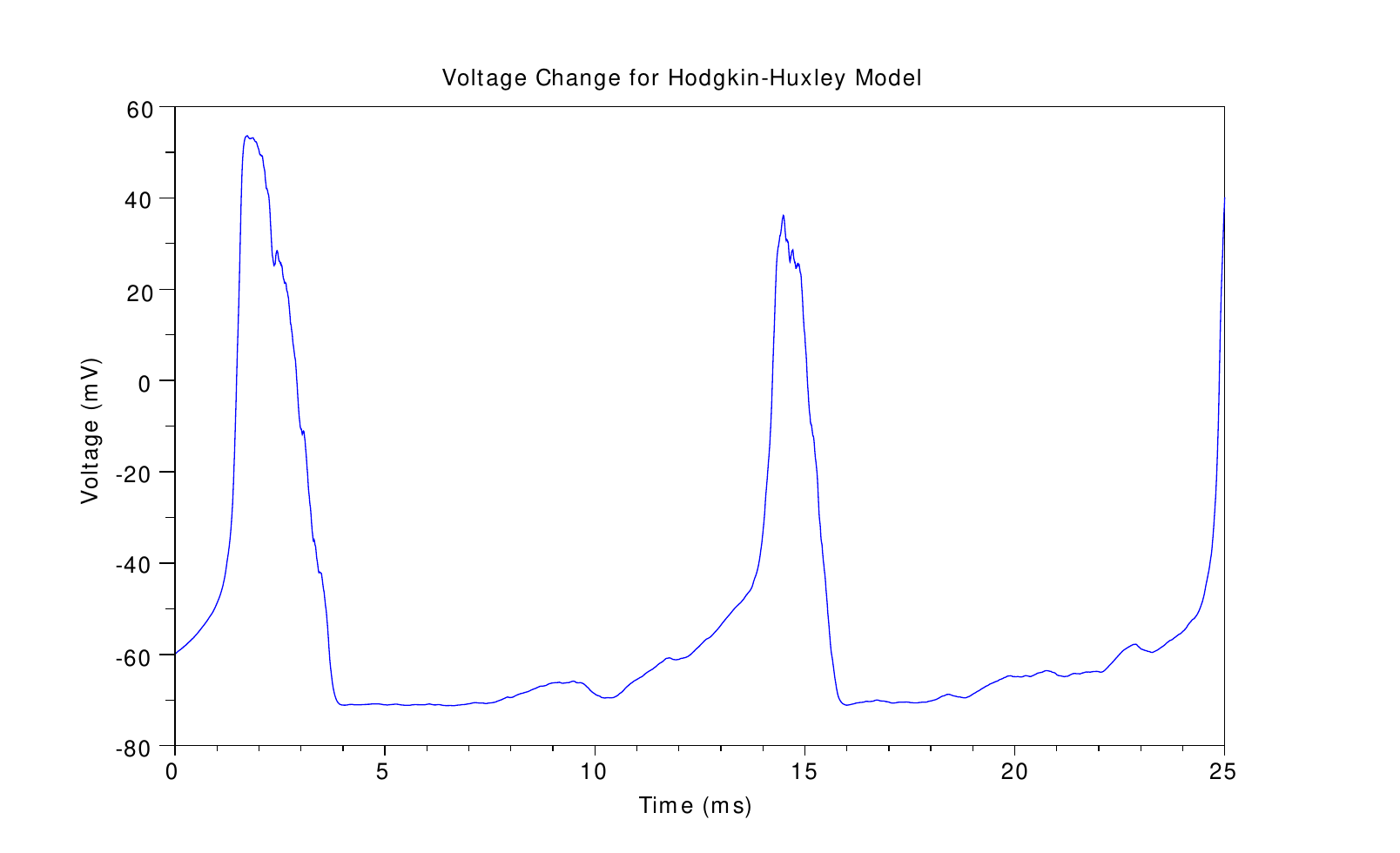}
                \caption{Voltage: $\sigma=0.5$}
        \end{subfigure}
\caption{Viable Stochastic Model: It\^o's Interpretation}\label{fig3}
\end{figure}

Our results in Section \ref{main-results} imply that the gating variables in the 
model (\ref{model3})
take undesired values outside of the unit interval. 
The simulations in Figure \ref{fig2}
illustrate this observation for different values of the parameter 
$\sigma=\sigma_i, i=1,2,3 $.  Here, we used the  
Euler-Maruyama method for the numerical implementation (see \cite{kloeden} and \cite{ruemelin}). 
We remark that It\^o's and Stratonovich's interpretation
yield the same solution for the stochastic model (\ref{model3}).
\\

\begin{figure}[h]
        \begin{subfigure}[b]{0.4\textwidth}
                \centering
                \includegraphics[width=7cm]{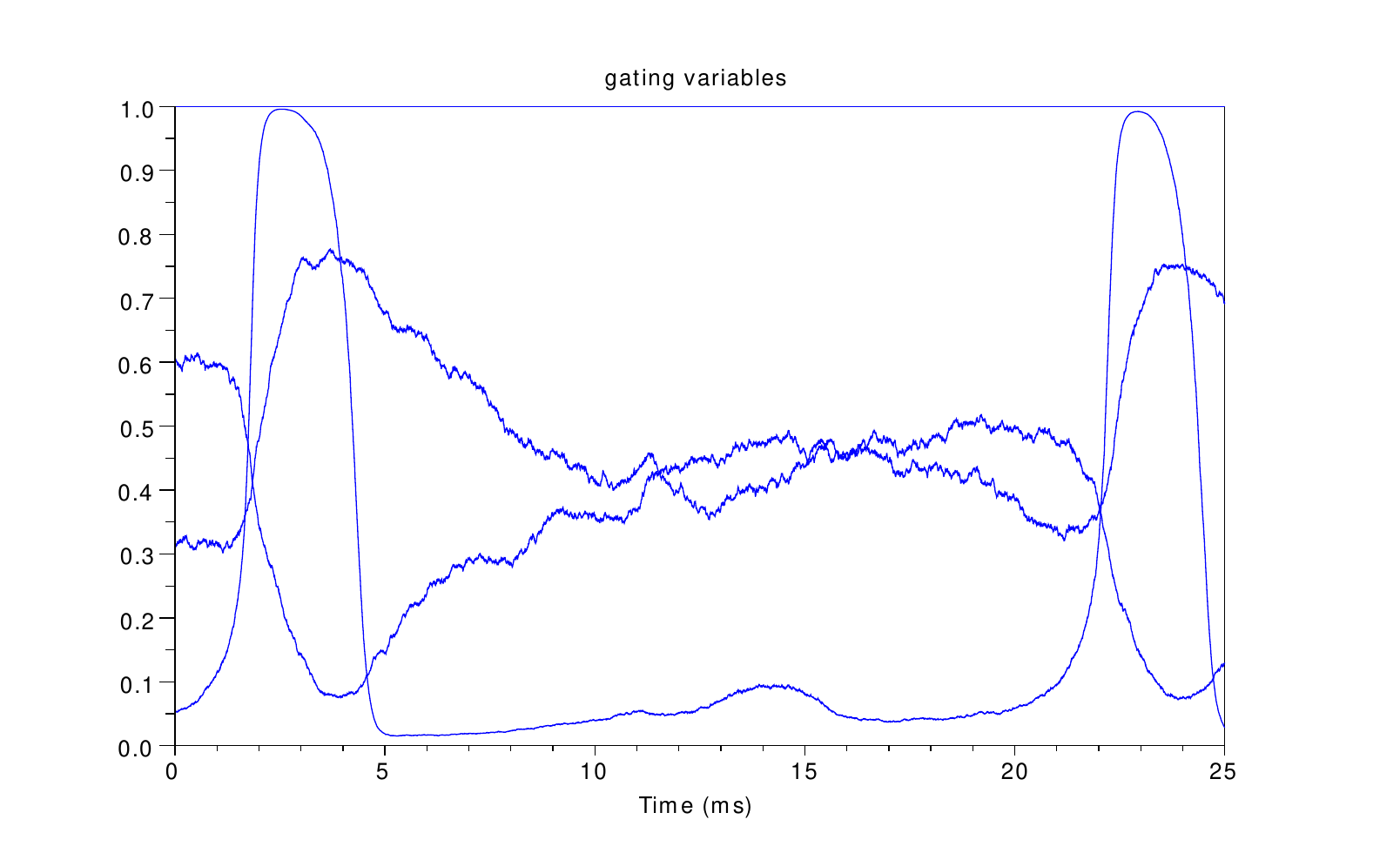}
                \caption{Gating Variables: $\sigma=0.1$}
        \end{subfigure}%
        \quad
             \begin{subfigure}[b]{0.4\textwidth}
                \centering
                \includegraphics[width=7cm]{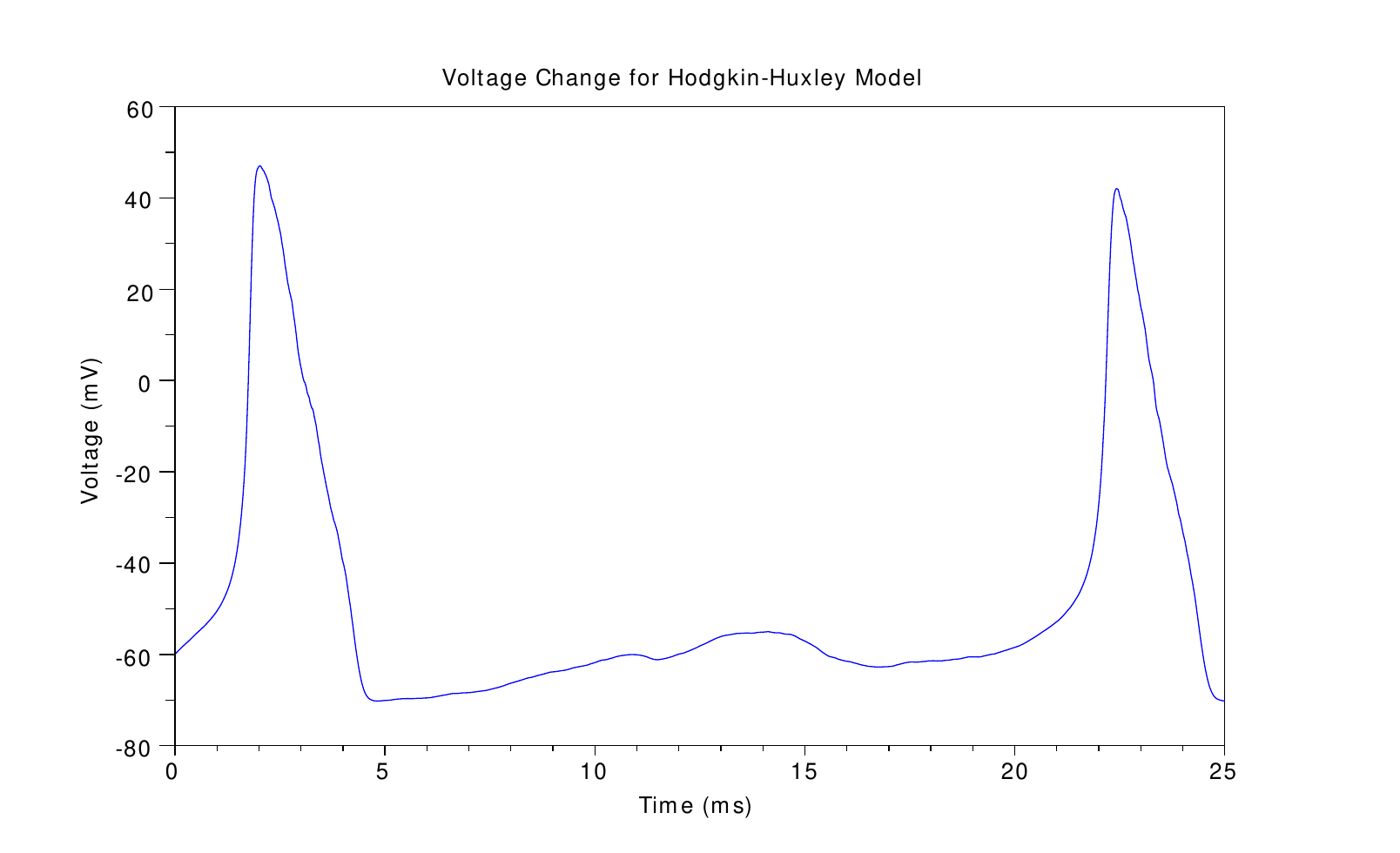}
                \caption{Voltage: $\sigma=0.1$}
        \end{subfigure}\\
     \begin{subfigure}[b]{0.4\textwidth}
                \centering
                \includegraphics[width=7cm]{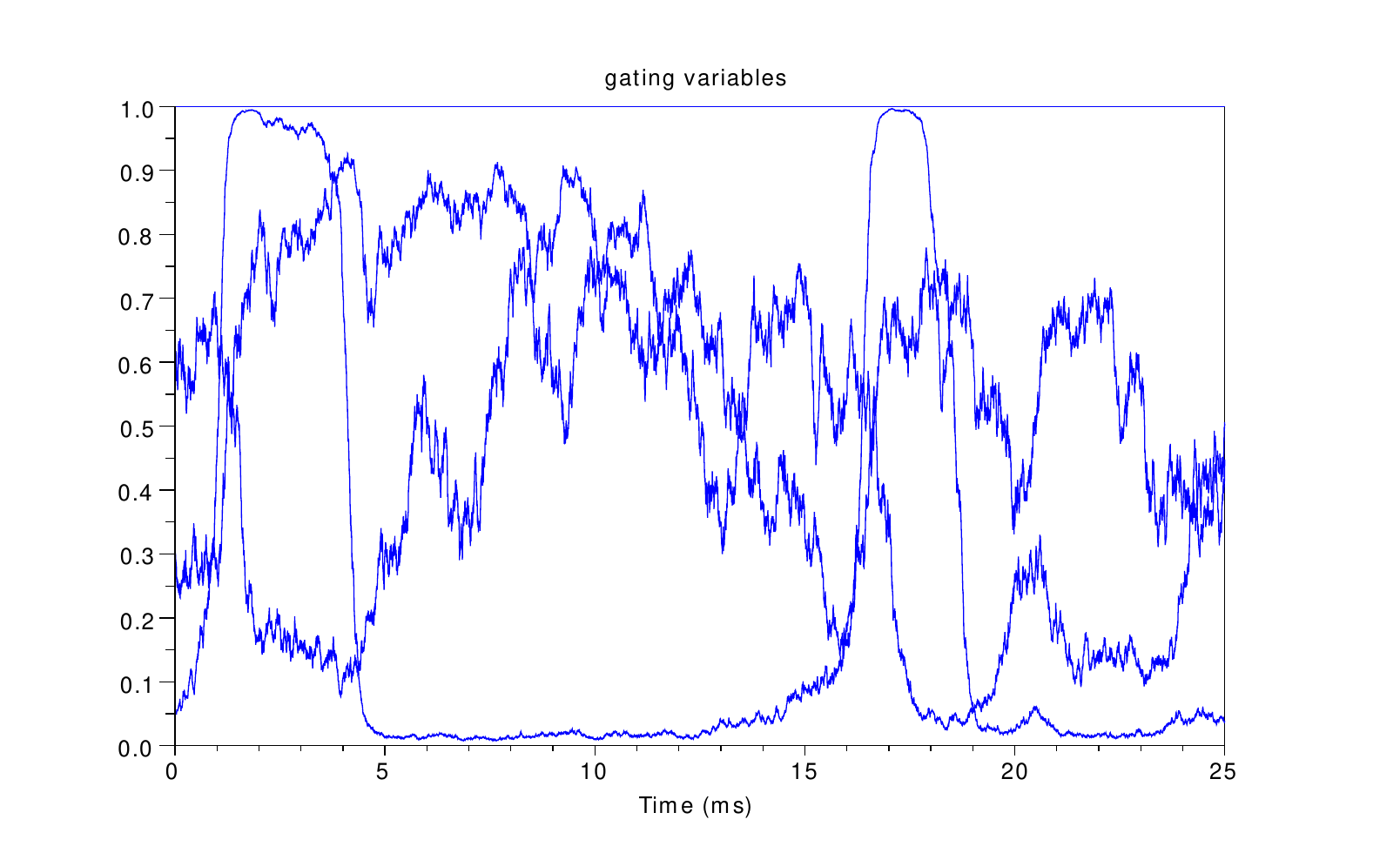}
                \caption{Gating Variables: $\sigma=0.5$}
                \end{subfigure}
    \quad
      \begin{subfigure}[b]{0.4\textwidth}
                \centering
                \includegraphics[width=7cm]{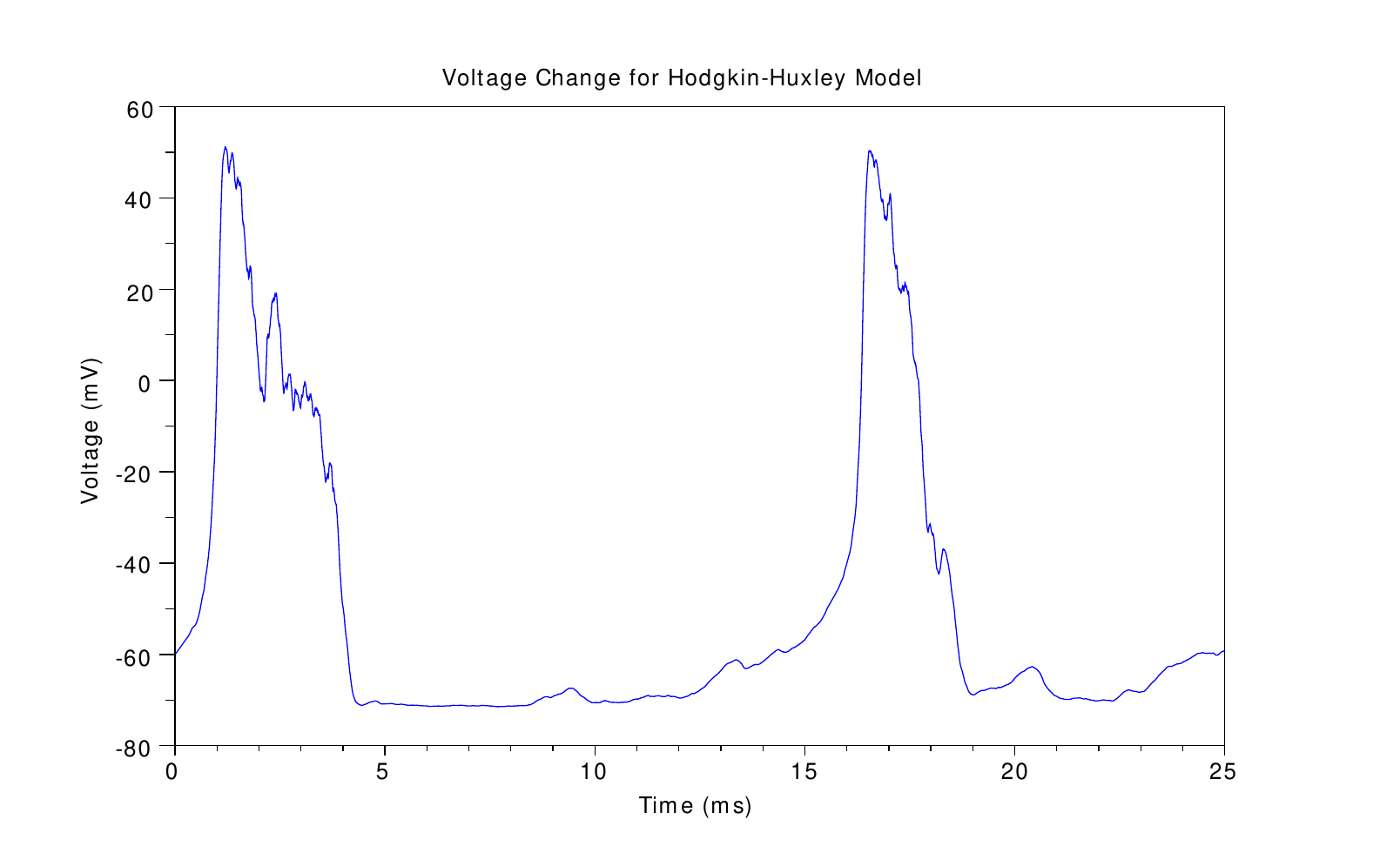}
                \caption{Voltage: $\sigma=0.5$}
        \end{subfigure}
\caption{Viable Stochastic Models: Stratonovich's Interpretation}\label{fig4}
\end{figure}

We obtain viable stochastic models if we
replace the constants $\sigma_i$ in the equations for the gating variables 
by appropriate functions $g_i$, $i=1,2,3$, that ensure the desired invariance of the unit interval.
To be more precise, we may consider stochastic models, where the determining equations 
for the gating variables $x_i$ are of the form
\begin{align}\label{model3}
dx_i(t) & =   (\alpha_i(V(t))(1-x_i(t))-\beta_i(V(t))x_i(t)) dt +g_i(t,x_1(t),x_2(t),x_3(t))\, dW_{i}(t),
\end{align}
for $i=1,2,3$, and the stochastic perturbations $g_i:[0,\infty[\times \R^3\rightarrow\R$ satisfy
\begin{align*}
g_i(t,x_1,x_2,x_3)=0\qquad\quad\textnormal{for}\  x_1,x_2,x_3\in[0,1]\ \textnormal{such that}\ x_i\in\{0,1\}, 
\end{align*}
for all $t\geq 0$ and $i=1,\dots,3$.
\\

Theorem \ref{mainB} in Section \ref{main-results} immediately implies that the gating variables in the 
modified stochastic model (\ref{model3}) take values within the interval $[0,1]$, and that it 
is valid for It\^o's and Stratonovich's interpretation of the stochastic differential equations.
\\

We illustrate the model behavior for stochastic perturbations
of the form
$$
g_i(t,x_1,x_2,x_3)=\sigma x_i (1-x_i) \qquad\ i=1,\dots,3, 
$$
where the constant $\sigma>0$.
The simulations in Figure \ref{fig3}
show the behavior of the solutions for It\^o's interpretation 
of the stochastic system (\ref{model3}), where we used
the Euler-Maruyama method for the numerical implementation.
Figure
\ref{fig4} illustrates the model behavior when we apply Stratonovich's interpretation.
In this case we applied the Euler-Heun method for the simulations (see \cite{kloeden} and \cite{ruemelin}).
\\

\section{Proof of the Theorems}\label{proofs}

Our proof is based on the main theorems obtained by A. Milian in \cite{milian-1995}.
We first recall her results, which are formulated  for systems of stochastic It\^o differential equations, 
and yield necessary and sufficient conditions for the stochastic viability of 
polyhedral subsets and the validity of comparison theorems.\\

\begin{definition}
A subset $K\subset \R^m$ is said to possess the \textbf{\textit{stochastic viability property}} with respect to 
the system $(f,g)$ if for every initial data $X_0\in K$ and every $t_0\geq 0$ there exists a global solution
of the initial value problem (\ref{stochastic}), and the solution satisfies
$$P\left(\left\{X(t)\in K,\ t\in[t_0,\infty[\right\}\right).
$$
\end{definition}
\vspace*{2ex}

For vectors $a,n\in\R^m$ we denote by 
$$H_{a,n}:=\{x\in \R^m,\  \langle x-a,n\rangle\geq 0\}
$$
the half-space determined by $a$ and $n$, where 
$\langle\cdot,\cdot\rangle$ is the scalar product in $\R^m$. A \textbf{\textit{polyhedron}} $K$ in $\R^m$ is a set of the form
$$K=\bigcap_{\nu\in I}H_{a_\nu,n_\nu},
$$
where $I=\{1,\dots,N\}\subset \N$ is a finite subset and $a_\nu,n_\nu\in\R^m$, $\nu\in I$.\\

For the proof of the following results we refer to \cite{milian-1995}.

\begin{theorem}\label{mil1}
Let $K=\bigcap_{\nu\in I}H_{a_\nu,n_\nu}$ be a polyhedron in $\R^m$ and suppose that the functions $f$ and $g$ satisfy the following conditions:
\begin{itemize}
\item[(a)] For every $T>0$ there exists a constant $C_T>0$ such that 
$$\|f(t,x)\|^2+\|g(t,x)\|^2\leq C_T(1+\|x\|^2)\qquad\textnormal{\emph{for all}}\ x\in K, t\in [0,T].
$$
\item[(b)] For every $T>0$ there exists a constant $\widetilde C_T>0$ such that 
$$\|f(t,x)-f(t,y)\|+\|g(t,x)-g(t,y)\|\leq \widetilde C_T\|x-y\| \qquad\textnormal{\emph{for all}}\ x,y\in K, t\in [0,T].
$$
\item[(c)] For every $x\in K$ the functions $f(\cdot, x)$ and $g(\cdot, x)$ are continuous on $[0,\infty[$.

\end{itemize}
Then, the set $K$ possesses the stochastic viability property with respect to the system $(f,g)$ if and only if for all $\nu\in I$ and $x\in K$
such that $\langle x-a_\nu,n_\nu\rangle=0$ we have
\begin{align*}
\,\langle f(x,t),n_\nu\rangle\,&\geq 0,\\
\langle g_j(x,t),n_\nu\rangle&= 0,\qquad  j=1,\dots,r,
\end{align*}
for all $t\geq 0$, where $g_j$ is the $j$-th column of the matrix $g=[g_{ij}]$.
\end{theorem}

\begin{theorem}\label{mil2}
Let $I$ be a non-empty subset of $\{1,\dots,m\}$ and suppose that
for every $T>0$ there exists a constant $C_T>0$
such that
\begin{itemize}
\item[(a)]
\ \ $\|f(t,x)\|^2+\|g(t,x)\|^2\leq C_T(1+\|x\|^2)\qquad\textnormal{\emph{for all}}\ x\in \R^m, t\in [0,T].$\\[.03ex]
\item[(b)]\ \ 
$\|f(t,x)-f(t,y)\|+\|g(t,x)-g(t,y)\|\leq  C_T\|x-y\| \qquad\textnormal{\emph{for all}}\ x,y\in \R^m, t\in [0,T].$\\[.03ex]
\item[(c)]\ For every $x\in K$ the functions $f(\cdot, x)$ and $g(\cdot, x)$ are continuous on $[0,\infty[$.

\end{itemize} 
We assume that the functions
$\overline{f}$ and $\overline{g}$ satisfy the same conditions and denote the corresponding solutions 
of the stochastic  systems $(f,g)$ and $(\overline{f},\overline{g})$ by $X$ and $Y$.
Then, the following statements are equivalent:
\begin{itemize}

\item[(i)\,] For all $t_0\geq 0$, $X_0=((X_0)_1,\dots,(X_0)_m)\in\R^m$ and $Y_0=((Y_0)_1,\dots,(Y_0)_m)\in\R^m$ 
such that $(X_0)_i\geq(Y_0)_i, i\in I$, the corresponding solutions satisfy
$$P(\{X_i(t)\geq Y_i(t),\ i\in I, t\geq t_0\})=1.
$$

\item[(ii)] For all $i\in I$ the functions $f$ and $g$ satisfy 
\begin{eqnarray*}
f_i(t,x)&\geq& \overline{f}_i(t,y)\,\qquad \textnormal{\emph{for}}\ t\geq 0, \\
g_{ij}(t,x)&=&\overline{g}_{ij}(t,y)\qquad \textnormal{\emph{for}}\ t\geq 0,\, j=1,\dots,r,
\end{eqnarray*}
and all $x=(x_1,\dots,x_m)\in\R^m$, $y=(y_1,\dots,y_m)\in\R^m$
such that $x_k\geq y_k$, $k\in I$, $x_i=y_i$.
\end{itemize}

\end{theorem}

We remark that the assumptions (a)-(c) in Theorem 
\ref{mil1} are imposed to guarantee the existence of solutions 
of the stochastic initial value problem $(f,g)$. The more restrictive hypothesis in 
Theorem \ref{mil2} imply the existence and uniqueness of solutions, and therefore, the
stochastic viability of a subset is equivalent to the stochastic invariance
with respect to the system $(f,g)$.
\\

We will deduce our criteria 
from Milian's result and show that they are valid 
independent of It\^o's and Statonovich's interpretation of  stochastic differential equations. 
In the sequel, we use 
the symbol $\circ\, dW(t)$ to indicate Stratonovich's interpretation.
For convenience of the reader we recall the 
general conversion formula for systems of stochastic differential equations, 
which relates both interpretations (see \cite{evans}, Section 6E):\\

If we interpret the stochastic system $(\ref{stochastic})$ in the sense of Stratonovich, 
that is, $X$ is the solution of the stochastic system
 $$dX(t)= f(t,X(t)) dt+g(t,X(t))\circ\,dW(t),
$$
then, $X$ solves the system of It\^o equations
$$dX(t)=\left[f(t,X(t))+\frac{1}{2}h(t,X(t))\right]dt+g(t,X(t))dW(t),
$$
where the function $h=[h_i]:[0,\infty[\times\R^m\rightarrow\R^m$ is given by
\begin{equation}\label{trans}
h_i(t,x)=\sum_{k=1}^r\sum_{j=1}^m\frac{\partial g_{ik}}{\partial x_j}(t,x)g_{jk}(t,x),\qquad i=1,\dots, m.
\end{equation}

\begin{proof}[Proof of Theorem 2]
Using polyhedral subsets of $\R^m$  the positive cone can be represented as 
$$K^+=\bigcap_{i=1}^mH_{0,e_i},
$$
where $0\in\R^m$ denotes the origin and $e_i, i=1,\dots,m$, the standard orthonormal basis vectors in  $\R^m$.
Since we a priori assume the existence and uniqueness of solutions of the stochastic initial value problem $(f,g)$, the stochastic viability 
of the positive cone $K^+$ is equivalent to its invariance with respect to the system $(f,g)$.
Evaluating the necessary and sufficient conditions formulated in Theorem \ref{mil1} immediately follows the
result for It\^o's interpretation of the stochastic system.
We need to show that the statement remains valid if we apply Stratonovich's interpretation of stochastic differential equations.
Let $X$ be a solution of the Stratonovich equation
$$dX(t)= f(t,X(t)) dt+g(t,X)\circ\,dW(t).
$$
Then, the transformation formula implies that $X$ solves the system of It\^o equations $(\hat f, g)$ with modified
interaction term $\hat f=f+\frac{1}{2}h$, where the function $h$ is defined by the formula (\ref{trans}).
We apply our previous result, which is valid for It\^o's interpretation, to the stochastic system $(\hat f,g)$ and conclude that 
the positive cone is an invariant subset if and only if 
\begin{eqnarray*}
\hat f_i (t,x) &\geq &0\qquad\quad x\in K^+\ \textnormal{such that}\ x_i=0,\\
g_{i,j} (t,x) &= &0 \qquad\quad  x\in K^+\ \textnormal{such that}\ x_i=0,\ j =1,\dots ,r,
\end{eqnarray*}
for all $t\geq 0$ and $i\in I,$. 
The conditions on the stochastic perturbations yield the representation
\begin{equation}\label{aux_g}
g_{i,j} (t,x) = x_i\int_0^1\frac{\partial g_{i,j}}{\partial x_i} (t,x_1,\dots,sx_i,\dots,x_m)ds \qquad\quad i\in I, =1,\dots,r,
\end{equation}
and it follows that the functions $f$ and $g$ satisfy the conditions in Theorem \ref{mainA} if 
and only if the functions $\hat f$ and $g$ fulfill these conditions.
This observation concludes the proof for Stratonovich's interpretation.
\end{proof}

\begin{proof}[Proof of Theorem 3]
We can represent the subset $K$ in Theorem \ref{mainB} as the finite intersection of 
polyhedral subsets
$$K=\bigcap_{i\in I}\big(H_{0,e_i}\cap H_{e_i,-e_i}\big).
$$
Computing explicitly the necessary and sufficient conditions for the invariance of the subset $K$
in Theorem \ref{mil1} follows the statement for the system of It\^o equations $(f,g)$.

To prove the result for Stratonovich's interpretation we 
use the explicit relation between both interpretations and the representation (\ref{aux_g})
for the stochastic perturbations in the final part of the proof of Theorem \ref{mainB}. 
This leads to the modified system of It\^o equations $(\hat f,g)$, for which necessary and sufficient 
conditions are known. We observe that
the conditions on the functions $f$ and $g$ are equivalent to the same conditions 
for the functions $\hat f$ and $g$, and are therefore invariant under the transformation
relating both interpretations.
\end{proof}

\begin{proof}[Proof of Theorem 4] The comparison theorem for It\^o's interpretation is 
valid by Theorem \ref{mil2}. To show the result
for Stratonovich's interpretation of stochastic differential equations we use the explicit 
transformation formula, which leads to the modified system of It\^o equations $(\hat f,g)$. 
We apply the known result for It\^o's interpretation
and observe that 
the conditions for the functions $\hat f$ and $g$ are equivalent to the conditions for the functions 
$f$ and $g$.
\end{proof}

\section*{Concluding Remarks}

We obtained necessary and sufficient conditions  for the invariance of rectangular subsets of the euclidean  space
under systems of stochastic differential 
equations and proved that 
the invariance property is independent of It\^o's and Stratonovich's interpretation.
In particular, we were able to characterize the class of stochastic perturbations 
that preserve the invariance property of the unperturbed deterministic system of
ODEs.
Such results are very relevant for applications and allow to validate stochastic models.

When not only temporal but also spatial properties are relevant,
the models are generally formulated as systems of stochastic PDEs. 
We are currently working on the extension of our invariance results for systems of parabolic PDEs under 
stochastic perturbations. 
A first result in this direction has been obtained in \cite{cresson-efendiev-sonner-2011}.
\\

\textbf{Acknowledgement:}\ The third author is funded by the ERC Advanced Grant FPT-246775 NUMERIWAVES.

\end{document}